\documentclass[12pt,leqno]{article} 

\usepackage{amsthm}
\usepackage{amsfonts}
\usepackage[reqno]{amsmath}
\usepackage{amssymb,comment}
\usepackage{color,psfrag,graphicx}

\usepackage{verbatim}

\newtheorem{satz}{Theorem}
\newtheorem{cor}{Corollary}
\newtheorem{rem}{Remark}
\newtheorem{defi}{Definition}
\newtheorem{lem}{Lemma}
 
\newtheorem{prop}{Proposition}

\def\N{{\mathbb N}}
\def\R{{\mathbb R}} 
\def\Rd{{\mathbb R}^d}
 
\def\Z{{\mathbb Z}}

\def\diam{{\rm diam \, }}
\def\dist{{\rm dist \, }}

\newcommand{\be}{\begin{equation}}

\newcommand{\ee}{\end{equation}}

\newcommand{\calk}{{\mathcal K}}
\DeclareMathOperator{\supp}{supp}

\allowdisplaybreaks

\title{Refined localization spaces, Kondratiev spaces with fractional smoothness and extension operators}

\author{Markus Hansen\thanks{
Markus Hansen,  Philipps-Universit\"at Marburg, Hans-Meerwein-Str. 6,  35032 Marburg; 
E-mail: {\sf markus.hansen1@gmx.net},
Phone: +49\,6421\,28\,254868, Fax: +49\,6421\,28\,26945.
Research supported by the European Research Council (ERC) 
under the grant StG306274.}\,   and Cornelia Schneider\thanks{Cornelia Schneider,  FAU Erlangen, Department Mathematik,  Cauerstr. 11,  91058 Erlangen; E-mail: {\sf schneider@math.fau.de},  Phone: +49\,9131\,85\,67207,  Fax: +49\,9131\,85\,67207.}
}

\date{\today}

\begin{document}

\maketitle

\begin{abstract}  
In this paper, we introduce Kondratiev spaces of fractional smoothness based on their close relation to refined localization spaces. Moreover, we investigate relations to other approaches leading to extensions of the scale of Kondratiev spaces with integer order of smoothness, based on complex interpolation, and give further results for complex interpolation of those function spaces. As it turns out to be one of the main tools in studying these spaces on domains of polyhedral type, certain aspects of the analysis of Stein's extension operator are revisited. Finally, as an application, we   study Sobolev-type embeddings.
\end{abstract}

\noindent
{\bf AMS Subject Classification:} 
41A25, 
41A46, 
41A65,  
42C40,  
65C99\\ 

\noindent
{\bf Key Words:} Kondratiev spaces, complex interpolation, extension operator,  fractional smoothness, refined localization spaces, Sobolev embedding.

\tableofcontents






\section{Introduction}

In recent years, Kondratiev spaces, a certain type of weighted Sobolev spaces, have become popular when  studying existence and regularity of solutions to partial differential equations on non-smooth domains -- of particular interest being elliptic problems on polygons, polyhedra, or diffeomorphic deformations thereof. The present article, as part of an ongoing research project \cite{smcw-a, smcw-b}, is another contribution to the investigation of basic properties of such function spaces.

Spaces of functions beyond classical (integer) orders of differentiation are of interest for a number of reasons. Let us just mention two of those which are clearly also relevant in connection with Kondratiev spaces. First, many partial differential equations allow a variational formulation from which existence of solutions in variants of spaces $H^1(D)$ can be derived -- this particularly applies to linear elliptic problems. However, for the calculation of approximate solutions, e.g. via finite element methods, a higher regularity is required. On the other hand, even for very basic problems like the Poisson equation on a non-convex polygon it is well-known that solutions generally do not belong to $H^2(D)$. Thus a finer distinction of regularity becomes necessary.

Another motivation for spaces of fractional smoothness stems from embedding assertions. The classical Sobolev embedding theorem states that the space $W^m_p(D)$ consists of continuous functions whenever $m>d/p$. In particular, for $D\subset\R^2$ we have $W^m_2(D)\hookrightarrow C(D)$ whenever $m\geq 2$. On the other hand, it is known that functions from $W^m_2(D)$ are
H\"older continuous, that is, $W^2_2(D)\hookrightarrow C^s(D)$ for every $s<1/2$. Hence more precise embedding relations require notions of fractional smoothness.

In this manuscript we are mainly  concerned with extending the scale of Kondratiev spaces, classically introduced for integer order of smoothness, to fractional/real smoothness parameters. Respective approaches have been known for some time, see e.g. \cite[Chapter 3]{Tr78}, \cite{lototsky}, or \cite{Ammann}. These were usually based on the complex method of interpolation. In the present work, we shall follow an alternative route, based on close relations of Kondratiev spaces to so-called refined localization spaces and then compare this approach to the previous approaches. 

The paper is organized as follows: In Section 2, we will provide the necessary information about Kondratiev spaces and domains of polyhedral type to the extent necessary for our subsequent arguments. In Section 3, we will discuss Stein's extension operator, as it will be a crucial tool to transfer interpolation results from spaces on $\R^d$ to spaces on domains. 
 Section 4 introduces refined localization spaces, and recalls their basic properties, particularly their characterization via suitable wavelet systems. Our main results then follow in Section 5, where the relation of the refined localization spaces to Kondratiev spaces is discussed, and then exploited in connection with complex interpolation. As a first application, in the final Section 6  the results are used to derive a Sobolev-type embedding theorem for Kondratiev spaces of fractional smoothness.



\subsection*{Notation}

We start by collecting some general notation used throughout the paper.  As usual, $\N$ stands for the set of all natural numbers, $\N_0=\mathbb N\cup\{0\}$, and  $\Rd$, $d\in\N$, is the
$d$-dimensional real Euclidean space with $|x|$, for $x\in\R^d$, denoting the Euclidean norm of $x$. Let $\N_0^d$, where
$d\in\N$, be the set of all multi-indices, $\alpha = (\alpha_1, \ldots,\alpha_d)$ with $\alpha_j\in\N_0$ and
$|\alpha|:=\sum_{j=1}^d\alpha_j$.  
The set $B_{\varepsilon}(x)$  is the open ball of radius $\varepsilon >0$ centered at $x$.  
We denote by  $c$ a generic positive constant which is independent of the main parameters, but its value may change from line to line. The expression $A\lesssim B$ means that $ A \leq c\,B$. If $A \lesssim B$ and $B\lesssim A$, then we   write
$A\sim B$.   
Given two quasi-Banach spaces $X$ and $Y$, we write $X\hookrightarrow Y$ if $X\subset Y$ and the natural embedding is bounded.  
A domain $\Omega$ is an open bounded set in $\R^d$.  The test functions on $\Omega$ are denoted by
$C_0^\infty(D)=\mathcal{D}(\Omega)$ and $\mathcal{D}'(\Omega)$ stands for the set of distributions on $\Omega$.  
Let $L_p(\Omega)$, $1\leq p\leq \infty$, be the usual Lebesque spaces on $\Omega$.

\section{Kondratiev spaces on domains of polyhedral type}

In this section,  for convenience to the reader,  we briefly recall the definition of domains of polyhedral type and some basic properties of related Kondratiev spaces central to our further arguments. For a more thorough treatment of these topics we refer to \cite{smcw-a} and the extended version \cite{smcw-a-ext}.  

\subsection{Domains of polyhedral type}

We begin with the definition of the domains under consideration. In the sequel the term {\em singularity set}, usually denoted by $S$, of a domain $\Omega\subset\R^d$ refers to the set of all points $x\in\partial\Omega$ such that for any $\varepsilon >0$ the set $\partial\Omega\cap B_{\varepsilon}(x)$ is not smooth.

We will mainly be interested in the case that  $d$ is either $2$ or $3$ and that  $\Omega$ is a bounded domain of polyhedral type. 
The precise definition will be given below in Definitions \ref{def:polyhedral-type-R2} and \ref{def:polyhedral-type}. Essentially, we will consider domains for which the analysis of the associated Kondratiev spaces can be reduced to the 
following four  basic cases:
\begin{itemize}
 \item Smooth cones;
 \item Specific nonsmooth cones;
 \item Specific dihedral domains;
 \item Polyhedral cones.
\end{itemize}
Let $d\ge 2$.
Below, an infinite  smooth  cone with vertex at the origin is the set
\[
 \{x\in \Rd: \: 0 < |x|< \infty\, , \: x/|x| \in \Omega\}\, , 
\]
where $\Omega$ is a simply connected subdomain of the unit sphere $S^{d-1}$ in $\Rd$ with $C^\infty$ boundary.\\
{~}\\

\noindent 
\begin{minipage}{0.4\textwidth}
{\bf Case I:} {\em Kondratiev spaces on smooth cones.}  
Let $K'$ be an infinite  smooth cone in $\R^d$ as described above, and let  $M:=\{0\}$.
Then we define the truncated cone $K$  by \\
 \begin{equation}\label{trunc}
K:= K'\cap B_1(0)
\end{equation}
and put
\end{minipage}\hfill \begin{minipage}{0.5\textwidth}
\includegraphics[width=6cm]{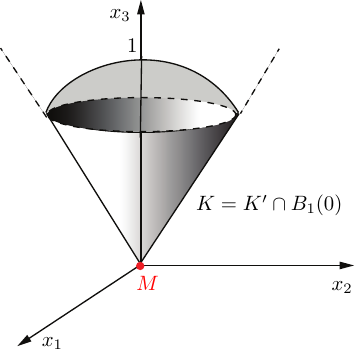}
\end{minipage}
\begin{equation} \label{smoothconekondratiev}
\|u|\calk^m_{a,p}(K,M)\| :=
 \Big(\sum_{|\alpha|\leq m}\int_{K}|\,  |x|^{|\alpha|-a}\partial^\alpha u(x)|^p\,dx\Big)^{1/p} \, .
\end{equation}
Observe that $M$ is just a part of the singular set of the boundary of {the truncated cone} $K$.  We further remark that the truncation in \eqref{trunc} not necessarily has to be done with the unit ball.  Any smooth hypersurface in $\mathbb{R}^d$,  particularly a hyperplane (see Case II below) would be sufficient (since it does not generate additional singularities).
\\

\noindent
{\bf Case II:} {\em Kondratiev spaces on specific nonsmooth cones}.\\
Let $K'$ denote a rotationally symmetric smooth cone with opening angle
$\gamma \in (0,\pi/2)$ (the precise value of $\gamma$ will be of no importance),
and let $K$ be its truncated version  $K:=K'\cap\{x\in \mathbb{R}^d:  0<x_d<1\}$  as in Case I.  Moreover, we put
\be\label{ws-01} 
	\hspace{-1cm}I :=\{x \in \R^d:~0<x_i<1,~i=1,\ldots,d\},
\ee
the unit cube in $\R^d$. Then we define two (non-diffeomorphic) versions of specific non-smooth
cones $P$:

\begin{minipage}{0.48\textwidth}
\includegraphics[width=6.5cm]{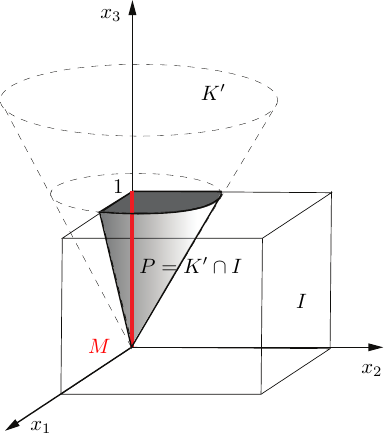}

{\bf Case IIa:} We consider $P=K'\cap I$, or\\ equivalently $P=K\cap I$.
\end{minipage}\hspace{\fill}
\begin{minipage}{0.48\textwidth}
\quad \includegraphics[width=6.5cm]{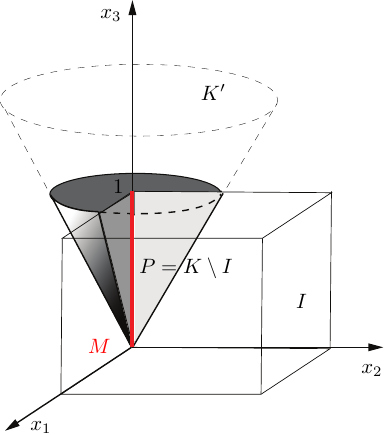}

{\bf Case IIb:} Alternatively, also $P=K\setminus\overline I$ will be called a specific non-smooth
	cone.
\end{minipage}\\[5mm]
In both situations, we choose
$$M=\Gamma : = \{x\in \R^d:~ x=(0,\ldots,0, x_d),~0 \le x_d \le 1\}$$
and define
\begin{equation} 
\|u|\calk^m_{a,p}(P,\Gamma)\|:=
\Big(\sum_{|\alpha|\leq m}\int_{P}|\,  \rho(x)^{|\alpha|-a}\partial^\alpha u(x)|^p\,dx\Big)^{1/p}\, , 
\end{equation}
where $\rho(x)$ denotes the distance of $x$ to $\Gamma$, i.e.,  $\rho (x) = |(x_1,\ldots,x_{d-1})|.$
Also in this case the set $\Gamma$ is a proper subset of the singular set of $P$.
\\

\noindent
\begin{minipage}{0.5\textwidth}
	{\bf Case III:}\\
	{\em Kondratiev spaces on specific dihedral domains}.\\
	Let $1\le \ell < d$ and let $I$ be the unit cube defined in \eqref{ws-01}.
	For $x \in \R^d$ we write 
	$$x=(x',x'') \in \R^{d-\ell}\times \R^{\ell},$$ 
	where $x':= (x_1, \ldots \, , x_{d-\ell})$ as well as
	\mbox{$x'':= (x_{d-\ell + 1}, \ldots \, , x_{d})$}. Hence 
	$I = I' \times I''$ with the obvious interpretation. 
	
	Additionally, also sets of the form $I=I'\times K$, where $K\subset\R^{d-\ell}$
	is a truncated cone as in Case I, are considered.
\end{minipage}\hfill
\begin{minipage}{0.4\textwidth}
\includegraphics[width=6cm]{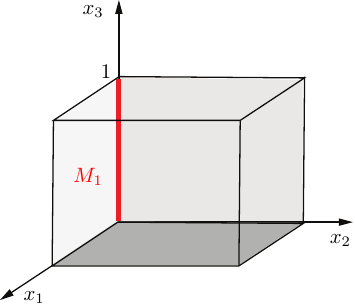}
\end{minipage}\\ \\ 

\noindent Then we choose
\be\label{ws-15}
	M_\ell := \{x \in \partial I:~x_1 = \ldots = x_{d-\ell }=0\}
\ee
and define
\begin{equation} \label{Kondratievdihedral}
\|u|\calk^m_{a,p}(I, M_\ell) \|
:= \Big(\sum_{|\alpha|\leq m}\int_{I}|\,  |x'|^{|\alpha|-a}\partial^\alpha u(x)|^p\,dx\Big)^{1/p}\, .
\end{equation}

\noindent
\begin{minipage}{0.5\textwidth}
{\bf Case IV:} {\em Kondratiev spaces on  polyhedral cones}.\\
	{Let $K'$ be an infinite cone in $\R^3$ with vertex at the origin, such that
	$\overline{K'}\setminus\{0\}$ is contained in the half space $\{x\in\R^3:\,x_3>0\}$.}
 We assume that the boundary $\partial K'$ 
consists of the vertex $x=0$, the edges (half lines) $M_1, \ldots \, , M_n$, and smooth faces $\Gamma_1, \ldots , \Gamma_n $.
This means $\Omega := K' \cap S^{2}$ is a domain of polygonal type on the unit sphere with sides 
$\Gamma_k\cap S^2$. Therein without loss of generality we may assume that $\Omega$ is simply connected.
We put 
\[
Q:= K' \,  \cap \, \{x \in \R^3: ~ 0 < x_3 < 1\}\, .
\]
\end{minipage}\hfill \begin{minipage}{0.45\textwidth}
\quad \includegraphics[width=6cm]{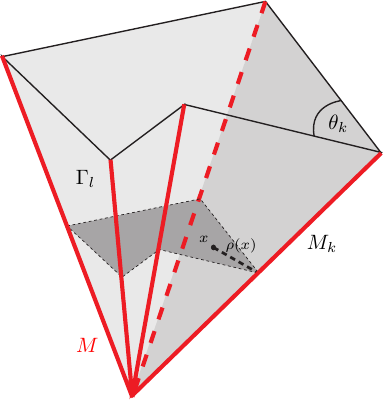}
\end{minipage}\\
In this case, we choose
$$M: = (M_1 \cup \ldots \cup M_n)\cap\overline{Q}$$
and define
\begin{equation} \label{def-norm-poly-cone}
\|u|\calk^m_{a,p}(Q,M)\|:=
\Big(\sum_{|\alpha|\leq m}\int_{Q}|\,  \rho(x)^{|\alpha|-a}\partial^\alpha u(x)|^p\,dx\Big)^{1/p}\, ,
\end{equation}
where $\rho(x)$ denotes the distance of $x$ to $M$. 
\\

\begin{rem}\label{rem:def-poly-cone}
	The technical assumption $\overline{K'}\setminus\{0\}\subset\{x\in\R^3:\,x_3>0\}$ immediately implies
	that the truncated cone $Q$ is bounded (alternatively we can truncate the cone $K'$ by intersecting
	with the unit ball $B_1(0)$). Moreover, if $\vartheta(x)$ denotes the angle between the line
	$\overrightarrow{0x}$ and the positive $x_3$-axis, then $\vartheta(x)\leq\vartheta_0<\pi/2$ for all
	$x\in Q$.
\end{rem}

Based on these four cases, we define the specific domains we will be concerned with in this paper.
\begin{defi} \label{def:polyhedral-type-R2}
	Let $D$ be a domain in $\R^2$ with singularity set $S$. 
	Then $D$ is of {\em polyhedral type}, if there exist finite disjoint index sets
	$\Lambda_1$ and $\Lambda_2$ and a covering $(U_i)_i$ of bounded open sets such that 
	$$\overline{D}
		\subset \Big(\bigcup_{i \in \Lambda_1} U_i\Big)
		\cup\Big( \bigcup_{j \in \Lambda_2} U_{j} \Big)\,,$$
	where 
	\begin{enumerate}
		\item
			for $i \in \Lambda_1$ the set $U_i$ is a ball such that $\overline{U_i} \cap S = \emptyset$;
		\item
			for $j\in \Lambda_2$ there exists a $C^{\infty}$-diffeomorphism 
			$\eta_{j}~:~\overline{U_{j}}\longrightarrow \eta_{j}(\overline{U_{j}})\subset \R^d$ such that
			$\eta_{j}(U_{j} \cap D)$ is a smooth cone $K$ as described in {\bf Case I}. Moreover, we assume
			that for all $x \in U_{j} \cap D$ the distance to $S$ is equivalent to the distance to the point
			$x^{j} := \eta_{j}^{-1}(0).$
	\end{enumerate}
\end{defi}

\begin{defi} \label{def:polyhedral-type} \label{def-domain}
	Let $D$ be a domain in $\R^d$,  $d\geq 3$,   with singularity set $S$. 
	Then $D$ is of {\em polyhedral type}, if there exist finite disjoint index sets
	$\Lambda_1,\ldots,\Lambda_5$ and a covering $(U_i)_i$ of bounded open sets such that 
	$$\overline{D}
		\subset \Big(\bigcup_{i \in \Lambda_1} U_i\Big)
		\cup\Big( \bigcup_{j \in \Lambda_2} U_{j} \Big) 
		\cup\Big( \bigcup_{k \in \Lambda_3} U_{k}\Big)
		\cup\Big( \bigcup_{\ell \in \Lambda_4} U_{\ell}\Big)
		\, ,$$
	where 
	\begin{enumerate}
		\item
			for $i \in \Lambda_1$ the set $U_i$ is a ball such that $\overline{U_i} \cap S = \emptyset$;
		\item
			for $j\in \Lambda_2$ there exists a $C^{\infty}$-diffeomorphism 
			$\eta_{j}~:~\overline{U_{j}}\longrightarrow \eta_{j}(\overline{U_{j}})\subset \R^d$ such that
			$\eta_{j}(U_{j} \cap D)$ is a smooth cone $K$ as described in {\bf Case I}. Moreover, we assume
			that for all $x \in U_{j} \cap D$ the distance to $S$ is equivalent to the distance to the point
			$x^{j} := \eta_{j}^{-1}(0).$
		\item
			for  $k \in \Lambda_3$ there exists a $C^{\infty}$-diffeomorphism
			$\eta_{k}~:~\overline{U_{k}} \longrightarrow \eta_{k}(\overline{U_{k}})\subset\R^d$ such that
			$\eta_{k}(U_{k}\cap D)$ is the nonsmooth cone $P$ as described in {\bf Case II}. Moreover, we
			assume that for all $x\in U_{k}\cap D$ the distance to $S$ is equivalent to the	distance to the
			set $\Gamma^{k}:= \eta_{k}^{-1}(\Gamma)$.
		\item
			for  $\ell \in \Lambda_{4}$   there exists a $C^{\infty}$-diffeomorphism 
			$\eta_{\ell}: \overline{U_{\ell}} \longrightarrow  \eta_{\ell}(\overline{U_{\ell}}) \subset\R^d$
			such that $\eta_{\ell}(U_{\ell}\cap D)$ is a specific dihedral domain as described in
			{\bf Case III}. Moreover, we assume that for all $x\in U_{\ell}\cap D$ the distance to $S$ is
			equivalent to the distance to the set $M^{\ell}:= \eta_{\ell}^{-1}(M_{n})$ for some
			$n \in \{1, \ldots \, , d-1\}$.
		\end{enumerate}

\noindent
In particular,  when $d=3$ we permit another type of subdomain: then 
$$ 
\overline{D}
	\subset \Big(\bigcup_{i \in \Lambda_1} U_i\Big)
		\cup\Big( \bigcup_{j \in \Lambda_2} U_{j} \Big) 
		\cup\Big( \bigcup_{k \in \Lambda_3} U_{k}\Big)
		\cup\Big( \bigcup_{\ell \in \Lambda_4} U_{\ell}\Big)
		\cup\Big( \bigcup_{m \in \Lambda_5} U_{m}\Big)\, ,
$$
where
\begin{itemize}
\item[v)]$ m \in \Lambda_{5}$  if there exists a $C^{\infty}$-diffeomorphism 
$\eta_{m}: \overline{U_{m}} \longrightarrow  \eta_{m}(\overline{U_{m}}) \subset\R^3$ such that
$\eta_{m}(U_{m}\cap D)$  is a polyhedral cone as described in {\bf Case IV}.  
Moreover, we assume that for all $x\in U_{m}\cap D$ the distance to  $S$ is equivalent to the distance to 
the set $M'_m := \eta_{m}^{-1}(M)$.

\end{itemize}	
			
\end{defi}

\begin{rem}\label{rem-def-domain}
	\begin{enumerate}
		\item
			Below we will not always distinguish between the cases $d=2$ and $d\geq 3$, as clearly the case $d=2$
			is similar	to the situation of a domain with conical  points in $\R^3$, that is, a domain with
			$\Lambda_3=\Lambda_4=\Lambda_5=\emptyset$.
		\item
			The Cases I--III are formulated for general $d$, and indeed, in the sequel our arguments 
			work in arbitrary dimensions. However,  let us point out here that  Definition \ref{def:polyhedral-type} for 
			$d>3$ leads to a very restricted class of domains (even the unit cube $[0,1]^d$ is not
			be included). To circumvent this problem would  require to consider further standard situations beyond Cases I--IV. Therefore, a  detailed discussion
			of higher-dimensional domains    is beyond the scope of this paper. In most numerical
			applications the case $d=3$ is the most interesting one anyway.
		\item
			In our considerations, domains of polyhedral type are always bounded, as they can be covered by
			finitely many bounded sets $U_i$. While a number of results can be extended to unbounded domains,
			in the sequel we will not discuss this case.
		\item
			In the literature many different types of polyhedral domains are considered.  
			For $d=3$  our definition coincides
			with the one of Maz'ya and Rossmann \cite[4.1.1]{MR2}.  
	\end{enumerate}
\end{rem}

\begin{rem}\label{rem-def-domain}
A discussion of a number of examples, as well as a slight generalization of Definition
	\ref{def:polyhedral-type} to include also certain non-Lipschitz domains which the na\"ive geometric
	intuition	would also label ``polyhedral domain'', can be found in \cite{smcw-a-ext}.
	Note that for each
	of those subdomains there is an unbounded and a bounded version (the latter being the intersection of the unbounded
	subdomain with a ball or a cube).
\end{rem}

\begin{rem}\label{rem-polyhedral-cone}
	While a cover of a polyhedral cone by specific non-smooth cones and dihedral domains can always be constructed by rather
	elementary means (see \cite[Lemma 2.7]{smcw-a}), such a cover is not sufficient for all purposes. More precisely,
	particularly in connection with extension arguments also a resolution of unity subordinate to such a cover comes into
	play, i.e. a family of smooth functions, pointwise adding to $1$, with compact supports inside the subdomains. As a
	consequence, while a polyhedral cone $D$ may be covered by finitely many non-smooth cones and dihedral domains there
	remains a neighbourhood of the vertex not covered by the supports of any member of the resolution of unity.
	
	For that reason the above definition explicitly includes polyhedral cones as an additional type of subdomain. As far as
	arguments for extension operators are concerned,  the results presented in Subsection \ref{sect-extension} are based on  a construction using an infinite
	number of specific subdomains (each of which in turn can be covered by finitely many dihedral domains). This is discussed in detail in \cite{HSS}. 
\end{rem}

\subsection{Kondratiev spaces}

While in the sequel we will only consider Kondratiev spaces on domains of polyhedral type, the actual definition can be formulated in slightly greater generality.

\begin{defi} \label{def-kondratiev}
	Let $\Omega$ be a  domain in $\R^d$, and let $M$ be a non-trivial closed subset of its boundary ${\partial \Omega}$.
	Furthermore, let $m\in\N_0$ and $a\in\R$. We put 
	\begin{equation}\label{gewicht} 
		\rho(x):= \min\{1, \dist(x,M)\} \, , \qquad x \in \Omega\, .
	\end{equation}
	{\rm (i)} Let $1\le p < \infty$.
	We define the Kondratiev spaces $\calk^m_{a,p}(\Omega, M)$ as the collection of all measurable functions which 
	admit $m$ weak derivatives in $\Omega$ satisfying
	$$\|u|\calk^m_{a,p}(\Omega,M)\|
		:= \Big(\sum_{|\alpha|\leq m}\int_\Omega|\rho(x)^{|\alpha|-a}\partial^\alpha u(x)|^p\,dx\Big)^{1/p}<\infty\,.$$
	{\rm (ii)} The space $\calk^m_{a,\infty}(\Omega, M)$ is the collection of all measurable functions which 
	admit $m$ weak derivatives in $\Omega$ satisfying
	$$\|u|\calk^m_{a,\infty}(\Omega,M)\|
		:= \sum_{|\alpha|\leq m} \|\, \rho^{|\alpha|-a} \partial^\alpha u \, |L_\infty (\Omega)\|<\infty\,.$$
\end{defi}
On occasion we shall additionally use the convention $\calk^0_{0,p}(\Omega,M)=L_p(\Omega)$.

\begin{rem} 
	{\rm (i)} Most often the set $M$ will be the {\em singularity set} $S$ of the domain $\Omega$. In that situation we shall
	use the abbreviated notation $\calk^m_{a,p}(\Omega)=\calk^m_{a,p}(\Omega,S)$.
	\\
	{\rm (ii)} Up to equivalence of norms, the weight function $\rho$ can be replaced by its regularized version
	$\varrho(x)=\psi(\delta(x))$, where $\psi$ is a smooth, strictly monotone function on $(0,\infty)$ with $\psi(t)=t$ for
	$t<1/2$ and $\psi(t)=1$ for $t\geq 2$, and $\delta$ is the regularized distance to $M$ as constructed in
	\cite[Chapter VI.1]{St70}. In the sequel, we will not strictly distinguish between $\dist(\cdot,M)$, $\rho$, $\delta$ and
	$\varrho$, i.e. by an abuse of notation we habitually refer to all versions by the symbol $\rho$ (it will always be clear
	from the context which precise version will be required).\\
	{\rm (iii)} The space 
	$$C^{\infty}_{\ast}(\Omega,M):=\{\varphi\big|_{\Omega}: \ \varphi \in C^{\infty}_0(\R^d\setminus M)\},$$
i.e.  smooth functions with compact support outside $M$,  is dense in $\calk^m_{a,p}(\Omega,S)$. A proof can be found in \cite{Tr85}. \\
		{\rm (iv)}
			In \cite{MR2} more general weighted Sobolev spaces
			on    polyhedral domains  are discussed. In particular, our spaces $\calk^m_{a,p}(D,S)$, where $D\subset \R^3$ is a domain of polyhedral type according to Definition \ref{def:polyhedral-type} with singularity set $S$,   coincide with the
			classes	$V^{\ell,p}_{\beta,\delta} (D)$ if $m= \ell$, 
			\[
				\beta = (\beta_1, \ldots\, , \beta_k)= (\ell-a, \ldots \, , \ell-a) 
				\quad \mbox{ and }\quad
				\delta = (\delta_1, \ldots\, , \delta_{k'}) = ( \ell - a, \ldots \, ,\ell - a )\, .
			\]
			For the meaning of $k$ and $k'$ we refer to \cite[4.1.1]{MR2}.
\end{rem}

A crucial tool in many investigations for Kondratiev spaces is the following simple shift operator, which yields an isomorphism linking spaces with different weight parameters.

\begin{prop}\label{prop-shift}
	For $b\in\R$ define the mapping
	$$T_b:u\mapsto\rho^b u.$$
	This operator is an isomorphism from $\calk^m_{a,p}(\Omega,M)$ onto $\calk^m_{a+b,p}(\Omega,M)$.
\end{prop}

The boundedness of $T_b$ follows by elementary calculations from the property
$$|\partial^\alpha\rho(x)|\leq c_\alpha\rho^{1-|\alpha|},\quad\alpha\in\N_0^d,$$
for the regularized distance $\rho$ (see \cite[Chapter VI.1]{St70}), and its inverse is given by $(T_b)^{-1}=T_{-b}$.

\subsection{Localization principles}

The following lemma allows to reduce the discussion of Kondratiev spaces on general domains of polyhedral type to any of the subdomains mentioned in Definitions \ref{def:polyhedral-type-R2} and  \ref{def-domain}.

\begin{lem}\label{lemma-deco}
	Let $D$, $(U_i)_i$, and $\Lambda_j$ with  $j=1, \ldots ,5$,  be as in Definitions  \ref{def:polyhedral-type-R2} and \ref{def-domain}. 
	Moreover, denote by $S$ the singularity set of $D$ and let $(\varphi_i)_i$ be a decomposition of unity subordinate to our
	covering, i.e., $\varphi_i \in C^\infty$, $\supp \varphi_i \subset U_i$, $0 \le \varphi_i \le 1$ and 
	$$\sum_i \varphi_i (x) = 1\qquad \mbox{for all}\quad x \in \overline{D}.$$
	We put $u_i := u \, \cdot \, \varphi_i$ in $D$.
	\\
	{\rm (i)} If $u\in \calk^m_{a,p}(D,S)$ then
	\begin{align*}
		\|u|\calk^m_{a,p}(D,S)\|^*
			&:=\max_{i\in\Lambda_1}\,\|u_i|W^m_{p}(D \cap U_i)\|
				+\max_{i\in\Lambda_2}\,\|u_i (\eta_i^{-1}(\,\cdot\,))|\calk^m_{a,p}(K,\{0\})\|\\
			&\quad +\max_{i\in\Lambda_3}\,\|u_i(\eta_i^{-1}(\,\cdot\,))|\calk^m_{a,p}(P,\Gamma)\|
				+\max_{i\in\Lambda_4}\,\|u_i(\eta_i^{-1}(\,\cdot\,))|\calk^m_{a,p}(I, M_\ell)\|\\
			&\quad +\max_{i\in\Lambda_5}\,\|u_i(\eta_i^{-1}(\,\cdot\,))|\calk^m_{a,p}(Q,M)\| 
	\end{align*}
	generates an equivalent norm on $\calk^m_{a,p}(D,S)$.
	\\
	{\rm (ii)}
	If $u:~D \to \mathbb{C}$ is a function such that the pieces $u_i$ satisfy
	\begin{itemize}
		\item
			$u_i \in W^m_{p}(D \cap U_i)$, $i\in \Lambda_1$;
		\item
			$u_i (\eta_i^{-1} (\, \cdot\, )) \in \calk^m_{a,p}(K,\{0\})$, $i\in \Lambda_2$;
		\item
			$u_i (\eta_i^{-1} (\, \cdot\, )) \in \calk^m_{a,p}(P, \Gamma)$, $i\in \Lambda_3$;
		\item
			$u_i (\eta_i^{-1} (\, \cdot\, )) \in \calk^m_{a,p}(I, M_\ell)$, $i\in \Lambda_4$;
		\item
			$u_i (\eta_i^{-1} (\, \cdot\, )) \in \calk^m_{a,p}(Q, M)$, $i\in \Lambda_5$;
	\end{itemize}
	then $u \in  \calk^m_{a,p}(D, S)$ and
	$$\|u|\calk^m_{a,p}(D,S)\|\lesssim\|u|\calk^m_{a,p}(D,S)\|^*\,.$$
\end{lem}

The following proposition will be fundamental to our further investigations.

\begin{prop}[{\cite[Proposition 2.11]{smcw-a}}]\label{prop-localization}
	Let $D\subset\R^d$ be a domain of polyhedral type. For sets
	$$
		D_j=\{x\in D:2^{-j-1}<\rho(x)<2^{-j+1}\},\quad j\in\N_0,
	$$
	we assume that there exists a resolution of unity $(\varphi_j)_{j\in\N_0}$ subordinate to the cover $(D_j)_{j\in\N_0}$ of
	$D$, i.e. $\varphi_j(x)\geq 0$ for all $x\in D$, $j\in\N_0$, $\supp\varphi_j\subset D_j$ and
	$\sum_{j=0}^\infty\varphi_j(x)=1$ for all $x\in D$. Moreover, we assume	$|\partial^\alpha\varphi_j(x)|\leq 2^{j|\alpha|}$
	for all $x\in D_j$ and all multiindices $\alpha\in\N_0^d$.
	
	Then for $a\in\R$, $m\in\N$ and $1<p<\infty$ we have
	\begin{equation}\label{eq-localization}
		\|u|\calk^m_{a,p}(D,M)\|^p
			\sim\sum_{j=0}^\infty\|\varphi_j u|\calk^m_{a,p}(D,M)\|^p,\quad u\in\calk^m_{a,p}(D,M)\,.
	\end{equation}
\end{prop}

For the existence of such a resolution of unity we refer to \cite[Proposition 17, Lemmas 18--19]{smcw-a-ext}. We particularly obtain

\begin{cor}\label{cor-localization}
	Let $D$ be any one of the unbounded versions of the subdomains mentioned in  Definitions \ref{def:polyhedral-type-R2},   \ref{def-domain},   and Remark
	\ref{rem-def-domain}, with singularity set $S$. Then for $D$ as well as $\Omega=\R^d\setminus S$ there exist corresponding
	resolutions of unity such that Proposition \ref{prop-localization} becomes applicable.
\end{cor}




\subsection{Extension operators}\label{sect-extension}

An indispensable tool in discussions of function spaces on domains are extension operators, i.e. bounded linear operators extending functions from the given domain $D\subset\R^d$ to the whole space $\R^d$. In the setting of unweighted and weighted Sobolev spaces, the extension operator defined by Stein in \cite[Chapter VI.3]{St70} is known to be a universal extension operator, i.e.   it is a bounded linear operator $\mathfrak{E}:W^m_p(D)\rightarrow W^m_p(\R^d)$ for arbitrary Lipschitz domains $D\subset\R^d$ and all parameters $m\in\N_0$ and $1<p<\infty$. In \cite{Hansen} this was extended to Kondratiev spaces. The core result can be formulated as follows:
\begin{prop}\label{prop-stein-kondratiev-1}
	Let $D\subset\R^d$ be a special Lipschitz-domain, i.e.
	$$D=\{x\in\R^d:x=(x',x_d), x'\in\R^{d-1},x_d>\omega(x')\}$$
	for some Lipschitz-continuous function $\omega:\R^{d-1}\rightarrow\R$. Further assume $\R^\ell_\ast:=\{x\in \R^d: \ x_{l+1}=\ldots=x_d=0\}\subset\partial D$,
	i.e. $\omega(x_1,\ldots,x_\ell,0,\ldots,0)=0$ for all $x_1,\ldots,x_\ell\in\R$. Then the Stein-extension operator
	$\mathfrak{E}$ on $D$ maps $\calk^m_{a,p}(D,\R^\ell_\ast)$ boundedly into $\calk^m_{a,p}(\R^d,\R^\ell_\ast)$.
\end{prop}

However, during discussions about the definition of the domains under consideration we noted that the subsequent extension to domains of polyhedral type proposed in \cite{Hansen} contained a gap. This extension uses the covering required in Definition \ref{def-domain} together with a subordinate decomposition of unity. However, as argued in Remark \ref{rem-polyhedral-cone}, in the special case of a polyhedral cone such a decomposition of unity involving finitely many smooth functions with compact support inside the subdomains cannot exist.

We addressed this gap in \cite{HSS} for $d=3$  in the even more general setting of the spaces $V^{m,p}_{\beta,\delta}(D)$ which contain the Kondratiev scale $\calk^m_{a,p}(D)$ as a special case.  In particular, \cite[Theorem~3.1]{HSS} gives the following result.

\begin{satz}\label{thm-stein-pcone}
	Let $K\subset\R^3$ be a (bounded or unbounded) polyhedral cone as in {\bf Case IV} with vertex in $0$. Let $a\in\R$, $m\in\N$,  and
	$1\leq p<\infty$. Then there exists	a universal bounded linear extension operator 
	$\mathfrak{E}:\calk^m_{a,p}(K,S)\rightarrow\calk^m_{a,p}(\R^3\setminus S,S)$, where $S$ is the singularity set of $K$.
\end{satz}

The notion of a universal extension operator refers to the fact that while the norm depends on $a$, $m$,  and $p$, the construction of the operator itself does not depend on these parameters.

Subsequently,  in \cite[Theorem~3.14]{HSS} it was also shown how to extend Theorem \ref{thm-stein-pcone}   from polyhedral cones to domains of polyhedral type  $D\subset \R^3$.  Together with what was outlined above, we are now even able to extend this further to general polyhedral domains in $\R^d$. 

\begin{satz}\label{thm-stein-general}
	Let $D\subset\R^d$ be a domain of polyhedral type in the sense of Definitions \ref{def:polyhedral-type-R2} and  \ref{def-domain} with singularity set $S$.
	Then there exists a universal bounded linear extension operator $\mathfrak{E}$ for $\calk^m_{a,p}(D,S)$,
	i.e. it maps $\calk^m_{a,p}(D,S)$ into $\calk^m_{a,p}(\R^d\setminus S,S)$ for all $m\in\N$, $1<p<\infty$ and $a\in\R$.
\end{satz}

\begin{proof}
	The results from Proposition \ref{prop-stein-kondratiev-1} and Theorem \ref{thm-stein-pcone} yield extensions for all
	types of subdomains referred to in Definitions \ref{def:polyhedral-type-R2} and  \ref{def-domain}. Clearly those results transfer also to smooth deformations
	of such domains, and hence together with Lemma \ref{lemma-deco} and yet another standard gluing argument as in \cite[Lemma~3.7]{HSS} we finally arrive at an extenstion operator for general domains of polyhedral type.
\end{proof}




\section{Refined localization spaces}

In this section we shall collect various definitions and properties for refined localization spaces used throughout this paper.  To a large extend this is based on the monographs \cite{Tr06, Tr08} where this scale of spaces was first introduced and investigated.

\subsection{Definition and basic properties}

The refined localization spaces were originally introduced in \cite{Tr06} when studying Triebel-Lizorkin spaces on domains with irregular boundary, with special focus on Lipschitz domains and domains with fractal boundary (like the Snowflake domain). The basic idea is to decompose a function or distribution using a resolution of unity subordinate to a Whitney decomposition of the domain.

Roughly spoken, a Whitney decomposition is a partition of a given domain into dyadic cubes in such a way that for any cube (with sufficiently small side length) its distance to the boundary of the domain is comparable to its side length. A description for a construction of such decompositions can be found in \cite[Chapter 6]{St70}.

\begin{defi}
	Let $Q_{j,k}=2^{-j}((0,1)^d+k)$ denote the open cube with vertex in $2^{-j}k$ and side length $2^{-j}$, $j\geq 0$,
	$k\in\Z$. Moreover, we denote by $2Q_{j,k}$ the cube concentric with $Q_{j,k}$ with side length $2^{-j+1}$. A
	collection of pairwise disjoint cubes $\{Q_{j,k_l}\}_{j\geq 0, l=1,\ldots,N_j}$ such that
	\begin{equation}\label{whitney-decomp}
		D=\bigcup_{j\geq 0}\bigcup_{l=1}^{N_j}\overline Q_{j,k_l}\,,\qquad
		\text{dist}(2Q_{j,k_l},\partial D)\sim 2^{-j}\,,\quad j\in\N\,,
	\end{equation}
	complemented by $\text{dist}(2Q_{0,k_l},\partial D)\geq c>0$ is called a Whitney decomposition of the domain $D$.
\end{defi}

Note that in what follows we can also replace such a Whitney decomposition by a cover with analogous properties, particularly when replacing cubes with vertices in $2^{-j}k$ by cubes centered in these points. Condition
\eqref{whitney-decomp} then ensures a uniformly bounded overlap of the cubes, i.e. the number of overlapping cubes at any given point $x\in D$ is uniformly bounded.

\begin{defi}\label{def-rloc}
	Let $\{\varphi_{j,l}\}$ be a resolution of unity of non-negative $C^\infty$-functions w.r.t. a  Whitney decomposition
	$\{Q_{j,k_l}\}$, i.e.
	\[
		\sum_{j,l}\varphi_{j,l}(x)=1\text{ for all }x\in D\,,\qquad
		|\partial^\alpha\varphi_{j,l}(x)|\leq c_\alpha 2^{j|\alpha|}\,,\quad\alpha\in\N_0^d\,.
	\]
	Moreover, we require $\supp\varphi_{j,l}\subset 2Q_{j,k_l}$. Then we define the refined localization spaces
	$F^{s,\mathrm{rloc}}_{p,q}(D)$ to be the collection of all locally integrable functions $f$ such that
	\[
		\|f|F^{s,\mathrm{rloc}}_{p,q}(D)\|
			=\biggl(\sum_{j=0}^\infty\sum_{l=1}^{N_j}\|\varphi_{j,l}f|F^s_{p,q}(\R^d)\|^p\biggr)^{1/p}<\infty\,,
	\]
	where $0<p<\infty$, $0<q\leq\infty$,  and $s>\sigma_{p,q}:=d\big(\frac{1}{\min(1,p,q)}-1\big)$.
\end{defi}

\begin{rem}\label{remark-rloc}
	A resolution of unity with the required properties can always be found: Start with a bump function
	$\phi\in C^\infty(\R^d)$ for $Q=(0,1)^d$, i.e. $\phi(x)=1$ on $Q$ and $\supp\phi\subset 2Q$. Via dilation and
	translation this yields bump functions $\phi_{j,k}$ for $Q_{j,k}$, with $\psi(x)=\sum_{j,k}\phi_{j,k}(x)>0$ for all 
	$x\in D$. The functions $\varphi_{j,k}=\phi_{j,k}/\psi$ have the required properties.
\end{rem}	

The original motivation to study refined localization spaces arose while investigating Triebel-Lizorkin spaces on domains with minimal regularity, particularly spaces
\[
	\widetilde F^s_{p,q}(D):=\{f\in F^s_{p,q}(\R^d):\supp f\subset\overline D\}\,.
\]
If $D$ is a bounded Lipschitz domain, then Triebel showed in \cite[Proposition 4.20]{Tr06} and \cite[Proposition 3.10]{Tr08} that for all parameters $0<p<\infty$, $0<q\leq\infty$ and $s>\sigma_{p,q}=d\big(\frac{1}{\min(1,p,q)}-1\big)$ these spaces actually coincide with the refined localization spaces. Such a simple characterization unfortunately fails for other types of domains, particularly for $D=\R^d\setminus\R^\ell_*$ in which we are interested here. On the other hand, in \cite[Theorem 3.19]{Scharf} explicit descriptions have been found for parameters $1\leq p,q<\infty$. For example, for smoothness $0<s<(d-\ell)/p$ we   have $F^{s,\mathrm{rloc}}_{p,q}(\R^d\setminus\R^\ell_*)=F^s_{p,q}(\R^d)$. Note however, that in this article we will primarily be interested in the case of smoothness $s>\sigma_{p,q}=d\big(\frac{1}{\min(1,p,q)}-1\big)$.

Another characterization for refined localization spaces gives a first hint for a connection to weighted spaces:
\begin{prop}
	Let $D\subset\R^d$ be an arbitrary domain,   $0<p<\infty$, $0<q\leq\infty$,  and $s>\sigma_{p,q}$. Then it holds
	\begin{equation}\label{eq-refined-weight}
		\|u|F^{s,\mathrm{rloc}}_{p,q}(D)\|\sim\|u|F^s_{p,q}(D)\|+\|\delta(\cdot)^{-s}u|L_p(D)\|\,,
	\end{equation}
	where $\delta(x)=\min\left(1,\dist(x,\partial D)\right)$. In particular, we always have
	\[
		\|u|F^s_{p,q}(D)\|\lesssim\|u|F^{s,\mathrm{rloc}}_{p,q}(D)\|\,,\quad\text{i.e.}\quad
		F^{s,\mathrm{rloc}}_{p,q}(D)\hookrightarrow F^s_{p,q}(D)\,.
	\]
\end{prop}
Regarding the proof for the case $q<\infty$ we refer to \cite[Proposition 3.24]{Scharf}.  A closer inspection of the proof in [32] reveals that the result remains true also for $q = \infty$. We refer to \cite[Rem.~2.3(iv)]{HansenScharf} for details.   Moreover, we note that in the above result we
can replace the distance function $\delta$   by the regularized distance $\rho$.




\subsection{Wavelet characterization}

As a further characterization we shall state here the description of the refined localization spaces in terms of appropriate wavelet systems.  
For many scales of function spaces wavelet characterizations allow to transfer problems from the function spaces to related problems for associated sequence spaces by utilizing corresponding isomorphisms, even though the wavelet systems not necessarily constitute bases. We intend to exploit this in the sequel. 

The definitions of $u$-wavelet systems,  the sequence spaces $f^s_{p,q}(\Z_\Omega)$  as well as other related  definitions and references can be found in the appendix.

\begin{satz}[{\cite[Theorem 2.38]{Tr08}}]\label{thm-wavelet}
	Let $\Omega\subsetneq\R^d$ be an arbitrary domain. Let
	$$0<p<\infty, \quad 0<q\leq\infty, \quad  s>\sigma_{p,q},$$
	and let $\{\Phi_{j,r}:j\in\N_0,r=1,\ldots,N_j\}$ be an orthonormal $u$-wavelet basis in $L_2(\Omega)$ with $u>s$.
	Finally, let
	$$\max(1,p)<v\leq\infty \qquad \text{and}\qquad s-\frac{d}{p}>-\frac{d}{v}.$$
	Then a function $f\in L_v(\Omega)$ belongs to the refined localization space $F^{s,\mathrm{rloc}}_{p,q}(\Omega)$ if, and
	only if, it can be represented as
	$$f=\sum_{j=0}^\infty\sum_{r=1}^{N_j}\lambda_{j,r}2^{-jd/2}\Phi_{j,r},\qquad\lambda\in f^s_{p,q}(\Z_\Omega).$$
	Furthermore, if $f\in F^{s,\mathrm{rloc}}_{p,q}(\Omega)$, then this representation is unique with $\lambda=\lambda(f)$, 
	$$\lambda_{j,r}(f)=2^{jn/2}\bigl(f,\Phi_{j,r}\bigr)=2^{jn/2}\int_\Omega f(x)\Phi_{j,r}(x)\,dx.$$
	Moreover, the mapping
	$$I:f\mapsto\lambda(f)=\bigl\{2^{jd/2}(f,\Phi_{j,r})\bigr\}$$
	is an isomorphism of $F^{s,\mathrm{rloc}}_{p,q}(\Omega)$ onto $f^s_{p,q}(\Z_\Omega)$; in particular, we have a
	quasi-norm equivalence
	$$\|f|F^{s,\mathrm{rloc}}_{p,q}(\Omega)\|\sim\|\lambda(f)|f^s_{p,q}(\Z_\Omega)\|.$$
\end{satz}

\begin{rem}
	Let us emphasize the fact that though not included in the above result, it remains valid also in
	case $s=0$ and $q=2$ upon defining
	$$F^{0,\mathrm{rloc}}_{p,2}(D)=L_p(D),\quad 1<p<\infty.$$
	This observation will allow us in the sequel to include interpolation results for the Banach couple
	$\{L_p(D),F^{s,\mathrm{rloc}}_{p,2}(D)\}$, $s>0$.
\end{rem}





\section{Complex interpolation of Kondratiev spaces}

In this section we introduce Kondratiev spaces of fractional smoothness order, based on their close relation to refined localization spaces as described below.  Afterwards we compare this approach to previous approaches   based on complex interpolation,  which also yield  an extension of the scale of Kondratiev spaces to fractional smoothness order. 

\subsection{A relation between Kondratiev and refined localization spaces}

The main tool in our investigation of complex interpolation for Kondratiev spaces will be the refined localization spaces from the previous section, as they are known to have a close relation with Kondratiev spaces. In particular, we have the following result from \cite{HansenScharf}:

\begin{satz}\label{thm-kond-rloc}
	Let $D\subset\R^d$ be a domain of polyhedral type with singular set $S\subset\partial D$, and consider
	$\Omega=\R^d\setminus S$. Let $1<p<\infty$ and $m\in\N$. Then we have
	\begin{equation}\label{eq-kond-rloc-1}
		\calk^m_{m,p}(\Omega)=F^{m,\mathrm{rloc}}_{p,2}(\Omega).
	\end{equation}
\end{satz}

This theorem allows to extend the previously used definition of Kondratiev spaces in two directions: For one, towards domains of the form $\Omega=\R^d\setminus\Gamma$ for arbitrary closed sets $\Gamma$ with $d$-dimensional Lebesgue-measure $|\Gamma|=0$ (which we won't do here), and two, towards spaces with fractional smoothness. The latter possibility will be discussed below, as far as it pertains to our current investigation.

Moreover, in view of Propositon \ref{prop-shift} we obtain the following: 

\begin{cor}
	For arbitrary $a\in\R$ we have
	\begin{equation}\label{eq-kond-rloc-2}
		\calk^m_{a,p}(\Omega)=T_{a-m}F^{m,\mathrm{rloc}}_{p,2}(\Omega).
	\end{equation}
\end{cor}

Hence instead of directly studying the behaviour of Kondratiev spaces under complex interpolation, we shall investigate the corresponding problem for refined localization spaces and subsequently transfer the obtained results back to Kondratiev spaces.


\subsection{Complex interpolation of refined localization spaces}

The key to study interpolation results for refined localization spaces are the wavelet characterizations from Theorem \ref{thm-wavelet}, specifically the wavelet isomorphism $I$, which allows to transfer the problem from the  function spaces $F^{s,\mathrm{rloc}}_{p,q}(\Omega)$ to a corresponding problem for the much easier-to-handle sequence spaces $f^s_{p,q}(\Z_\Omega)$, where as before $\Z_\Omega$ is some approximate lattice in $\Omega$.

The interpolation of sequence spaces is a well-established problem in interpolation theory, we refer to  e.g.  
\cite[Section 1.18]{Tr78}. In our context, interpolation results for sequence spaces $f^s_{p,q}(\Z_\Omega)$ can  easily be obtained via the method of retractions and coretractions from corresponding results for the Bochner spaces
$L_p(\Omega;\ell^s_q(\nabla_\Omega))$, where $\nabla_\Omega=\{(j,r):j\in\N,r=1,\ldots,N_j\}$, 
$$\|\lambda|\ell_q^s(\nabla_\Omega)\|^q=\sum_{j=1}^\infty\sum_{r=1}^{N_j}2^{jsq}|\lambda_{j,r}|^q\,,
	\qquad\lambda=\{\lambda_{j,r}\}_{(j,r)\in\nabla_\Omega}\,,$$
and for a general (quasi-)Banach space $A$ the Bochner space $L_p(\Omega)=L_p(\Omega;A)$ is defined via the norm
$$\|f|L_p(\Omega;A)\|=\int_\Omega\|f(\omega)|A\|^p\,d\omega\,,\qquad f\in L_p(\Omega;A)\,.$$
In particular, the space $L_p(\Omega;\ell^s_q(\nabla_\Omega))$ can be interpreted to contain sequences of functions
$f=\{f_{j,r}\}_{(j,r)\in\nabla_\Omega}$.

In turn, interpolation results for Bochner spaces $L_p(A)$ are well-known in the literature, we refer to
\cite[Section 1.18.4]{Tr78} and the references given therein. For the complex method of interpolation, it holds
$$\bigl[L_{p_0}(A_0),L_{p_1}(A_1)\bigr]_\Theta=L_p\bigl([A_0,A_1]_\Theta\bigr),\qquad
	\frac{1}{p}=\frac{1-\Theta}{p_0}+\frac{\Theta}{p_1},$$
for arbitrary $0<\Theta<1$, $1<p_0,p_1<\infty$ and a Banach space couple $\{A_0,A_1\}$. An analogous statement holds for $\ell_q$-spaces of vector-valued sequences, where we even may allow each component of a sequence to take values in a different Banach space; i.e.,  let $\mathcal{A}=(A_j)_{j\in\N}$ be a family of Banach spaces, we consider sequences $\lambda=(\lambda_j)_{j\in\N}$ with $\lambda_j\in A_j$. Then we get the interpolation result
$$\bigl[\ell_{q_0}(\mathcal{A}),\ell_{q_1}(\mathcal{B})\bigr]_\Theta=\ell_q([\mathcal{A},\mathcal{B}]_\Theta),
	\qquad\frac{1}{q}=\frac{1-\Theta}{q_0}+\frac{\Theta}{q_1},$$
where $\mathcal{B}=(B_j)_{j\in\N}$ is another family of Banach spaces, and $[\mathcal{A},\mathcal{B}]_\Theta$ denotes the corresponding family of interpolation spaces $([A_j,B_j]_\Theta)_{j\in\N}$. In case of spaces $\ell_q^s(\nabla_\Omega)$ we need to apply this interpolation result with $A_j=2^{js_0}\ell_{q_0}^{N_j}$ and $B_j=2^{js_1}\ell_{q_1}^{N_j}$ (i.e. $N_j$-dimensional $\ell_q$-spaces, whose norm is modified by a constant factor).

We now collect the relevant interpolation results:
\begin{prop}\label{prop-interpol-rloc}
	Let $0<\Theta<1$, $s_0s_1\in\R$,  $1\leq p_0,p_1,q_0,q_1\leq \infty$, and determine $p$ and $q$ as above. Then it
	holds
	\begin{enumerate}
		\item
			For the $\ell_q^s$-sequence spaces we have
			$$\bigl[\ell_{q_0}^{s_0}(\nabla_\Omega),\ell_{q_1}^{s_1}(\nabla_\Omega)\bigr]_\Theta
				=\ell_q^s(\nabla_\Omega),\qquad s=(1-\Theta)s_0+\Theta s_1.$$
			Consequently, it follows
			$$\bigl[L_{p_0}(\ell_{q_0}^{s_0}(\nabla_\Omega)),L_{p_1}(\ell_{q_1}^{s_1}(\nabla_\Omega))\bigr]_\Theta
				=L_p\bigl(\ell_q^s(\nabla_\Omega)\bigr).$$
		\item
			The method of retractions and coretractions then implies
			$$\bigl[f^{s_0}_{p_0,q_0}(\Z_\Omega),f^{s_1}_{p_1,q_1}(\Z_\Omega)\big]_\Theta
				=f^s_{p,q}(\Z_\Omega).$$
			Consequently, the wavelet isomorphism yields
			$$\bigl[F^{s_0,\mathrm{rloc}}_{p_0,q_0}(\Omega),F^{s_1,\mathrm{rloc}}_{p_1,q_1}(\Omega)\big]_\Theta
				=F^{s,\mathrm{rloc}}_{p,q}(\Omega).$$
	\end{enumerate}
\end{prop}

Note that the last statement particularly (and trivially) applies to the case $p_0=p_1=p$ and $q_0=q_1=q$, i.e. we have
\begin{equation}\label{eq-interpol-easy}
	\bigl[F^{s_0,\mathrm{rloc}}_{p,q}(\Omega),F^{s_1,\mathrm{rloc}}_{p,q}(\Omega)\big]_\Theta
	=F^{s,\mathrm{rloc}}_{p,q}(\Omega).
\end{equation}


\subsection{Complex interpolation of Kondratiev spaces}

The interpolation results from Proposition \ref{prop-interpol-rloc} and the relation \eqref{eq-kond-rloc-1} from Theorem \ref{thm-kond-rloc} now serve as the basis in order to obtain  interpolation results for Kondratiev spaces.

We begin with our definition for Kondratiev spaces of fractional smoothness, motivated by
\eqref{eq-kond-rloc-1} and \eqref{eq-kond-rloc-2}:

\begin{defi}\label{def-kond-fractional}
	Let $s\in\R_+$, $1<p<\infty$ and $a\in\R$. Furthermore, let $D\subset\R^d$ be a domain of polyhedral type with singular
	set	$S\subset\partial D$, and $\Omega=\R^d\setminus S$. Then we define
	\begin{equation}\label{eq-frac-kond-1}
		\mathfrak{K}^s_{s,p}(\Omega):=F^{s,\mathrm{rloc}}_{p,2}(\Omega)
	\end{equation}
	as well as
	\begin{equation}\label{eq-frac-kond-2}
		\mathfrak{K}^s_{a,p}(\Omega):=T_{a-s}\mathfrak{K}^s_{s,p}(\Omega).
	\end{equation}
\end{defi}

\begin{rem}
	Here and in what follows we always restrict ourselves to non-negative real parameters $s\in\R_+$. In the literature,
	Kondratiev spaces with negative integer smoothness have been introduced via duality. We shall not follow this approach
	here, and instead leave the study of dual spaces of Kondratiev and refined localization spaces and their complex
	interpolation to future research.
	
	Let us also stress the fact that in our definition the case $s=0$ is included via
	$\mathfrak{K}^0_{0,p}(\Omega)=F^{0,\mathrm{rloc}}_{p,2}(\Omega)=L_p(\Omega)$.
\end{rem}

As an easy consequence of \eqref{eq-interpol-easy} and the interpolation property for isomorphisms, we see that this version of Kondratiev spaces with fractional smoothness can also be obtained from complex interpolation.

\begin{lem}
	Let $s\in\R_+$ and $m_1,m_2\in\N_0$ with $m_1<s<m_2$.  Furthermore, let $0<\Theta<1$ be such that $s=(1-\Theta)m_1+\Theta m_2$  
	and let $1<p<\infty$, $a\in\R$. Then we have
	\begin{equation}\label{eq-frac-kond-3}
		\mathfrak{K}^s_{a,p}(\Omega)
			=\bigl[\calk^{m_1}_{a+m_1-s,p}(\Omega),\calk^{m_2}_{a+m_2-s,p}(\Omega)\bigr]_\Theta.
	\end{equation}
	In particular, for $s\not\in\N_0$, let $m=[s]$ be the integer part of $s$, and $\Theta=\{s\}=s-[s]$ its fractional part.
	Then
	\begin{equation*}
		\mathfrak{K}^s_{a,p}(\Omega)
			=\bigl[\calk^m_{a-\Theta,p}(\Omega),\calk^{m+1}_{a+1-\Theta,p}(\Omega)\bigr]_\Theta.
	\end{equation*}
\end{lem}

This is a direct consequence of \eqref{eq-interpol-easy}, applied with $q=2$. We only note that
\eqref{eq-frac-kond-3} can be immediately reduced to the case $a=s$ by applying $T_{a-s}$ to both sides. In the next subsection we will demonstrate that our definition essentially coincides with the other standard approach towards (weighted) Sobolev spaces of fractional smoothness -- via a slightly different application of complex interpolation.

\begin{rem}
	Apart from our approach via the wavelet isomorphism, interpolation results for refined localization spaces can also be
	obtained via a standard retraction/coretracton argument, based directly on Definition \ref{def-rloc}. Then interpolation
	results are traced back to ones for vector-valued sequence spaces $\ell_p(F^s_{p,q}(\R^d))$.
\end{rem}


\subsection{An alternative approach to Kondratiev spaces with fractional smoothness}

In the literature, whenever they are needed, Kondratiev spaces with fractional smoothness parameters most often are introduced via complex interpolation, analogously to widely used definitions for Sobolev spaces of fractional smoothness.

\begin{defi}
	Let $D\subset\R^d$ be a domain of polyhedral type with singularity set $S$, and put $\Omega=\R^d\setminus S$. Furthermore,
	let $s\in\R_+\setminus\N$,  $1<p<\infty$, and  $a\in\R$. Put $m=[s]$ and $\Theta=\{s\}$.
	Then we define
	$$\widetilde{\mathfrak{K}}^s_{a,p}(\Omega)
		=\bigl[\calk^m_{a,p}(\Omega),\calk^{m+1}_{a,p}(\Omega)\bigr]_\Theta.$$
\end{defi}

Note that in contrast to \eqref{eq-frac-kond-3} here all spaces involved have the same weight parameter $a$.

We now aim at proving that, up to equivalent norms, these two versions of Kondratiev spaces with fractional smoothness coincide. The key tool will be the localization property from Proposition \ref{prop-localization}, together with the following simple lemma.

\begin{lem}
	Let $\{A,B\}$ be an arbitrary Banach couple, and let $\alpha,\beta>0$. Then it holds for $0<\Theta<1$
	$$[\alpha A,\beta B]_\Theta=\alpha^{1-\Theta}\beta^\Theta [A,B]_\Theta$$
	with equality of norms. Therein, for some Banach space $V$ and a positive real number $\gamma$ the space $\gamma V$
	refers to the set $V$ equipped with a scaled version of the norm of $V$, i.e. $\|v|\gamma V\|=\gamma\|v|V\|$ for all
	$v\in V$.
\end{lem}

\begin{proof}
	This fact can be easily seen directly from the definition of the complex method: Without going into details, every function
	$f\in\mathcal{F}(A,B)$ can be transformed into a corresponding function $g\in\mathcal{F}(\alpha A,\beta B)$ with equal
	norm by putting $g(z)=\alpha^{\theta-1}\beta^{-\theta}f(z)$.
	
	To avoid having to recall the various definitions,  let's instead give a short proof based on the interpolation property.
	It is clear from the definition that the identity operator has norms
	$$\|\text{Id}:A\rightarrow\alpha A\|=\alpha\qquad\text{as well as}\qquad\|\text{Id}:B\rightarrow\beta B\|=\beta.$$
	Then the interpolation property yields that $\text{Id}$ maps $[A,B]_\Theta$ into $[\alpha A,\beta B]_\Theta$ with norm
	at most $\alpha^{1-\Theta}\beta^\Theta$, or equivalently that $\text{Id}$ maps
	$\alpha^{1-\Theta}\beta^\Theta [A,B]_\Theta$ into $[\alpha A,\beta B]_\Theta$ with norm at most $1$. The reverse
	inequality now can be obtained from the obvious identification $A=\alpha^{-1}(\alpha A)$, and thus ultimately
	$$\|u|[\alpha A,\beta B]_\Theta\|=\alpha^{1-\Theta}\beta^\Theta\|u|[A,B]_\Theta\|
		=\|u|\alpha^{1-\Theta}\beta^\Theta [A,B]_\Theta\|.$$
	This proves the claim.
\end{proof}

\begin{satz}\label{thm-equivalence}
Let $D\subset \mathbb{R}^d$ be a domain of polyhedral type with singularity set $S$,  and put $\Omega=\R^d\setminus S$.  
	Moreover, let $s\in\R_+$, $1<p<\infty$, and $a\in\R$. Then
	it holds
	$$\widetilde{\mathfrak{K}}^s_{a,p}(\Omega)=\mathfrak{K}^s_{a,p}(\Omega)$$
	in the sense of equivalent norms.
\end{satz}

\begin{proof}
	We first transfer the localization property from Proposition \ref{prop-localization} to both scales of interpolation
	spaces, and subsequently use the previous lemma to show that these scales actually coincide.
	
	{\bf Step 1:} Let $(\varphi_j)_{j\geq 0}$ be a partition of unity as in Corollary \ref{cor-localization}, such that
	Proposition \ref{prop-localization}	becomes applicable.	Moreover, for $m\in\N$ we put
	$$\Phi:\calk^m_{a,p}(\Omega)\rightarrow\ell_p\bigl(\calk^m_{a,p}(\Omega)\bigr),\qquad
		u\mapsto (\varphi_j u)_{j\geq 0}.$$
	By Proposition \ref{prop-localization} this mapping is a bounded linear injection; in terms of abstract interpolation
	theory, $\Phi$ is a coretraction. Put $\psi_j=\varphi_{j-1}+\varphi_j+\varphi_{j+1}$, $j\geq 1$ and
	$\psi_0=\varphi_0+\varphi_1$. Then
	$$\Psi:\ell_p\bigl(\calk^m_{a,p}(\Omega)\bigr)\rightarrow\calk^m_{a,p}(\Omega),\qquad
		(u_j)_{j\geq 0}\mapsto\sum_{j\geq 0}\psi_j u_j,$$
	is a corresponding retraction, i.e. it is a bounded linear operator, and we have $\Psi\circ\Phi=\text{Id}$. The latter
	identity immediately follows from the observation $\psi_j(x)=1$ for all $x\in\supp\varphi_j$ for all $j\geq 0$, and
	$\sum_{j\geq 0}\varphi_j(x)=1$ for all $x\in\Omega$. It remains to show the boundedness of $\Psi$.
	
	To see this we first note that for every $x\in\Omega$, at most four of the functions $\psi_j$ are non-vanishing, since
	$\supp\psi_j\subset\{x\in\Omega:2^{-j-2}\leq\rho(x)\leq 2^{-j+2}\}$. This property is commonly referred to as uniformly
	bounded overlap. We further observe that the bound on the derivatives of the functions $\varphi_j$ transfers to one for
	$\psi_j$, i.e. we have $|\partial^\alpha\psi_j(x)|\leq\widetilde c_\alpha 2^{j|\alpha|}$ for all $j\geq 0$,
	$x\in\supp\psi_j$,  and all multiindices $\alpha\in\N_0^d$. Altogether this yields
	\begin{align*}
		\|\Psi\{u_j\}_{j\geq 0}|\calk^m_{a,p}(\Omega)\|^p
			&=\int_\Omega\sum_{|\alpha|\leq m}\biggl|
				\rho(x)^{|\alpha|-a}\partial^\alpha\biggl(\sum_{j\geq 0}\psi_j(x)u_j(x)\biggr)\biggr|^p\,dx\\
			&\sim\sum_{j\geq 0}\int_{\supp\psi_j}\sum_{|\alpha|\leq m}\bigl|
				\rho(x)^{|\alpha|-a}\partial^\alpha\bigl(\psi_j(x)u_j(x)\bigr)\bigr|^p\,dx\\
			&\sim\sum_{j\geq 0}\int_{\supp\psi_j}\sum_{|\alpha|\leq m}2^{-j(|\alpha|-a)p}\biggl|
				\sum_{\beta\leq\alpha}\partial^\beta\psi_j(x)\partial^{\alpha-\beta}u_j(x)\biggr|^p\,dx\\
			&\lesssim\sum_{j\geq 0}\int_{\supp\psi_j}\sum_{|\alpha|\leq m}2^{-j(|\alpha|-a)p}
				\sum_{\beta\leq\alpha}2^{j|\beta|p}\bigl|\partial^{\alpha-\beta}u_j(x)\bigr|^p\,dx\\
			&\sim\sum_{j\geq 0}\int_{\supp\psi_j}\sum_{|\alpha|\leq m}
				\sum_{\beta\leq\alpha}\Bigl|\rho^{|\alpha-\beta|-a}\partial^{\alpha-\beta}u_j(x)\Bigr|^p\,dx\\
			&\lesssim\sum_{j\geq 0}\int_\Omega\sum_{|\alpha|\leq m}\Bigl|\rho^{|\alpha|-a}\partial^\alpha u_j(x)\Bigr|^p\,dx\\
			&=\bigl\|\{u_j\}_{j\geq 0}\big|\ell_p\bigl(\calk^m_{a,p}(\Omega)\bigr)\|^p.
	\end{align*}
	This shows the boundedness of the map $\Psi$.
	
	{\bf Step 2:} With the help of these maps we can transfer interpolation results for spaces $\ell_p(A)$ or
	$\ell_p(\mathcal{A})$ (where as  before $\mathcal{A}=(A_j)_{j\geq 0}$ is a family of Banach spaces) to interpolation
	results for Kondratiev spaces. More precisely, the abstract theory now guarantees
	$$\bigl\|u\big|\bigl[\calk^{m_0}_{a_0,p_0}(\Omega),\calk^{m_1}_{a_1,p_1}(\Omega)\bigr]_\Theta\bigr\|
		\sim\bigl\|\Phi u\big|\big[\ell_{p_0}\bigl(\calk^{m_0}_{a_0,p_0}(\Omega)\bigr),
			\ell_{p_1}\bigl(\calk^{m_1}_{a_1,p_1}(\Omega)\bigr)\bigr]_\Theta\bigr\|$$
	for all admissible choices of parameters.
	
	Upon noting that for $p_0=p_1=p$ we have
	$$\big[\ell_p\bigl(\calk^{m_0}_{a_0,p}(\Omega)\bigr),\ell_p\bigl(\calk^{m_1}_{a_1,p}(\Omega)\bigr)\bigr]_\Theta
		=\ell_p\Bigl(\bigl[\calk^{m_0}_{a_0,p}(\Omega),\calk^{m_1}_{a_1,p}(\Omega)\bigr]_\Theta\Bigr),$$
	with suitable choices for $m_0$, $m_1$, $a_0$ and $a_1$ (according to the respective definitions for Kondratiev spaces
	with fractional smoothness) we now immediately conclude that Proposition \ref{prop-localization}
	transfers to both $\mathfrak{K}^s_{a,p}(D)$ and $\widetilde{\mathfrak{K}}^s_{a,p}(D)$.
	
	{\bf Step 3:} We once again recall that by construction the weight function $\rho$ is essentially constant on
	$\Omega_j\supset\supp\varphi_j$, i.e. $\rho(x)\sim 2^{-j}$ for all $x\in\Omega_j$. As a particular
	consequence we now obtain
	$$\|\varphi_j u|\calk^m_{a+b,p}(\Omega)\|
		\sim 2^{jb}\|\varphi_j u|\calk^m_{a,p}(\Omega,\Omega_j)\|$$
	for all $b\in\R$. We conclude
	$$\|u|\calk^m_{a+b,p}(\Omega)\|
		\sim\bigl\|\Phi u\big|\ell_p\bigl(\calk^m_{a+b,p}(\Omega)\bigr)\bigr\|
		\sim\Bigl\|\Phi u\Big|\ell_p\Bigl(\bigl(2^{jb}\calk^m_{a,p}(\Omega)\bigr)_{j\geq 0}\Bigr)\Bigr\|.$$
	The interpolation property now yields
	\begin{align*}
		\|u|\widetilde{\mathfrak{K}}^s_{a,p}(\Omega)\|
			&\sim\bigl\|\Phi u|\big|\ell_p\bigl(\widetilde{\mathfrak{K}}^s_{a,p}(\Omega)\bigr)\bigr\|\\
			&=\Bigl\|\Phi u\Big|\ell_p\Bigl(\bigl([\calk^m_{a,p}(\Omega),\calk^{m+1}_{a,p}(\Omega)]_\Theta\bigr)_{j\geq 0}
				\Bigr)\Bigr\|\\
			&=\Bigl\|\Phi u\Big|\ell_p\Bigl(\bigl([2^{-j\Theta}\calk^m_{a,p}(\Omega),
				2^{j(1-\Theta)}\calk^{m+1}_{a,p}(\Omega)]_\Theta\bigr)_{j\geq 0}\Bigr)\Bigr\|\\
			&\sim\Bigl\|\Phi u\Big|\ell_p\Bigl(\bigl([\calk^m_{a-\Theta,p}(\Omega),
				\calk^{m+1}_{a+1-\Theta,p}(\Omega)]_\Theta\bigr)_{j\geq 0}\Bigr)\Bigr\|\\
			&=\bigl\|\Phi u|\big|\ell_p\bigl(\mathfrak{K}^s_{a,p}(\Omega)\bigr)\bigr\|\sim\|u|\mathfrak{K}^s_{a,p}(\Omega)\|,
	\end{align*}
	for all $u\in C_*^\infty(\Omega,S)$, i.e.,   smooth functions with compact support outside $S$.  But since  this space is  dense in all spaces
	involved this already proves the claim.
\end{proof}

As a consequence of this theorem,  from now on we will always use the notation $\mathfrak{K}^s_{a,p}(\Omega)$.


\subsection{A general interpolation result for Kondratiev spaces}

With the help of a reasoning similar to the proof of Theorem \ref{thm-equivalence} we can now present an interpolation result for arbitrary pairs of Kondratiev spaces, within the full scale of parameters $s$ and $a$.

\begin{satz}\label{thm-general-interpol}
	Let $\Omega=\R^d\setminus S$, where $S$ is the singularity set of some domain $D$ of polyhedral type.  Further, let $1<p<\infty$, $s_0,s_1\in\R_+$,  and $a_0,a_1\in\R$.
	Then it holds
	$$\mathfrak{K}^s_{a,p}(\Omega)
		=\bigl[\mathfrak{K}^{s_0}_{a_0,p}(\Omega),\mathfrak{K}^{s_1}_{a_1,p}(\Omega)\bigr]_\Theta,$$
	where
	$$s=(1-\Theta)s_0+\Theta s_1\qquad\text{and}\qquad a=(1-\Theta)a_0+\Theta a_1\,.$$
\end{satz}

\begin{proof}
	With essentially the same steps as before, we find
	\begin{align*}
		\|u|\mathfrak{K}^s_{a,p}(\Omega)\|
			&\sim\bigl\|\Phi u|\big|\ell_p\bigl(\mathfrak{K}^s_{a,p}(\Omega)\bigr)\bigr\|\\
			&=\Bigl\|\Phi u\Big|\ell_p\Bigl(\bigl([\mathfrak{K}^{s_0}_{a+s_0-s,p}(\Omega),
				\mathfrak{K}^{s_1}_{a+s_1-s,p}(\Omega)]_\Theta\bigr)_{j\geq 0}\Bigr)\Bigr\|\\
			&\sim\Bigl\|\Phi u\Big|\ell_p\Bigl(\bigl([2^{j(s_0-s)}\mathfrak{K}^{s_0}_{a,p}(\Omega),
				2^{j(s_1-s)}\mathfrak{K}^{s_1}_{a,p}(\Omega)]_\Theta\bigr)_{j\geq 0}\Bigr)\Bigr\|\\
			&=\Bigl\|\Phi u\Big|\ell_p\Bigl(\bigl([\mathfrak{K}^{s_0}_{a,p}(\Omega),
				\mathfrak{K}^{s_1}_{a,p}(\Omega)]_\Theta\bigr)_{j\geq 0}\Bigr)\Bigr\|\\
			&=\Bigl\|\Phi u\Big|\ell_p\Bigl(\bigl([2^{j(a-a_0)}\mathfrak{K}^{s_0}_{a,p}(\Omega),
				2^{j(a-a_1)}\mathfrak{K}^{s_1}_{a,p}(\Omega)]_\Theta\bigr)_{j\geq 0}\Bigr)\Bigr\|\\
			&\sim\Bigl\|\Phi u\Big|\ell_p\Bigl(\bigl([\mathfrak{K}^{s_0}_{a_0,p}(\Omega),
				\mathfrak{K}^{s_1}_{a_1,p}(\Omega)]_\Theta\bigr)_{j\geq 0}\Bigr)\Bigr\|\\
			&=\bigl\|\Phi u|\big|\ell_p\bigl([\mathfrak{K}^{s_0}_{a_0,p}(\Omega),
					\mathfrak{K}^{s_1}_{a_1,p}(\Omega)]_\Theta\bigr)\bigr\|
				\sim\bigl\|u|\big|[\mathfrak{K}^{s_0}_{a_0,p}(\Omega),\mathfrak{K}^{s_1}_{a_1,p}(\Omega)]_\Theta\bigr\|
	\end{align*}
	for all $u\in C_*^\infty(\Omega,S)$.  The last step once again is a consequence of the interpolation property applied to the
	operator $\Phi$.
\end{proof}

Let us explicitly note an important special case.

\begin{cor} Under the assumptions of Theorem \ref{thm-general-interpol}
	it holds
	$$\mathfrak{K}^s_{a,p}(\Omega)
		=\bigl[\mathfrak{K}^{s_0}_{a,p}(\Omega),\mathfrak{K}^{s_1}_{a,p}(\Omega)\bigr]_\Theta$$
	for arbitrary 
	$a\in \R$. In particular, for parameters $m,m_0,m_1\in\N$ and $0<\Theta<1$ such that
	$m=(1-\Theta)m_0+\Theta m_1$ we obtain
	$$\mathcal{K}^m_{a,p}(\Omega)
		=\bigl[\mathcal{K}^{m_0}_{a,p}(\Omega),\mathcal{K}^{m_1}_{a,p}(\Omega)\bigr]_\Theta$$
	in the sense of equivalent norms.
\end{cor}

The latter statement follows immediately upon recalling $\calk^m_{a,p}(\Omega)=\mathfrak{K}^m_{a,p}(\Omega)$ for all real $a$ and integer $m$.

As a consequence of this corollary we also find the following easy embedding result
\begin{lem}\label{lemma-embedding-vertical}
Let $\Omega=\R^d\setminus S$, where $S$ is the singularity set of some domain $D$ of polyhedral type. 	
	Moreover,  let $a\in\R$, $s_0\geq s_1$,  and $1<p<\infty$. Then it holds
	$$\mathfrak{K}^{s_0}_{a,p}(\Omega)\hookrightarrow\mathfrak{K}^{s_1}_{a,p}(\Omega).$$
\end{lem}

\begin{proof}
	We first note that for arbitrary ordered Banach couples $\{A,B\}$, i.e. where $A\hookrightarrow B$, we also have
	$$A\hookrightarrow [A,B]_\Theta\hookrightarrow [A,B]_{\Theta'}\hookrightarrow B$$
	for all $0<\Theta\leq\Theta'<1$.
	
	Next, the case $s_1=0$ can be dealt with easily: By applying the isomorphism $T_a$ the embedding is equivalent to the 
	case $a=0$, which is obvious in view of $\mathfrak{K}^0_{0,p}(\Omega)=L_p(\Omega)$.
	
	For the general case, choose integers $0\leq m_1<s_1\leq s_0<m_0$ and parameters $\Theta_0$ and $\Theta_1$ such that
	$s_i=(1-\Theta_i)m_0+\Theta_i m_1$, $i=0,1$. It follows $\Theta_0\leq\Theta_1$, and hence
	$$
		\mathfrak{K}^{s_0}_{a,p}(\Omega)
			=\bigl[\calk^{m_0}_{a,p}(\Omega),\calk^{m_1}_{a,p}(\Omega)\bigr]_{\Theta_0}
			\hookrightarrow\bigl[\calk^{m_0}_{a,p}(\Omega),\calk^{m_1}_{a,p}(\Omega)\bigr]_{\Theta_1}
			=\mathfrak{K}^{s_1}_{a,p}(\Omega)\,,
	$$
	as claimed.
\end{proof}


\subsection{Kondratiev spaces with fractional smoothness on domains of polyhedral type}

So far we always considered spaces on unbounded domains $\Omega=\R^d\setminus S$, which serve as the counterpart of the domain $\R^d$ for the classical unweighted Sobolev and Bessel-potential spaces $W^m_p(\R^d)$ and $H^s_p(\R^d)$. We shall now transfer the results for Kondratiev spaces on $\Omega$ to corresponding spaces on (bounded) domains $D$.

Recall that Lemma \ref{lemma-deco} allows to reduce all considerations in this section to one of the specific subdomains described in Definitions \ref{def:polyhedral-type-R2} and \ref{def-domain}, and as in Corollary \ref{cor-localization} we may focus on the unbounded versions. Thus,  in what follows $D$ always refers to an unbounded radially-symmetric cone, a non-smooth cone, a dihedral domain or a polyhedral cone.

Next, for any of those specific domains Theorem \ref{thm-stein-general} yields a bounded linear extension operator
$$\mathfrak{E}:\calk^m_{a,p}(D)\rightarrow\calk^m_{a,p}(\Omega)=\mathfrak{K}^m_{a,p}(\Omega).$$
The interpolation property then guarantees that the operator is also bounded w.r.t. the interpolation spaces,
$$\mathfrak{E}:\bigl[\calk^{m_0}_{a+m_0-s,p}(D),\calk^{m_1}_{a+m_1-s,p}(D)\bigr]_\Theta
	\rightarrow\bigl[\calk^{m_0}_{a+m_0-s,p}(\Omega),\calk^{m_1}_{a+m_1-s,p}(\Omega)\bigr]_\Theta
	=\mathfrak{K}^s_{a,p}(\Omega),$$
where $s=(1-\Theta)m_0+\Theta m_1$. Note in particular that the target space does not depend on the choice of $m_0$ and $m_1$. A direct approach to the interpolation spaces on the left then would depend on being able to argue that also those interpolation spaces do not depend on the choice of $m_0$ and $m_1$.

Here we shall go another route and instead (inspired by the introduction of Bessel-potential spaces $H^s_p(D)$ on domains) define Kondratiev spaces on domains $D$ by restriction.

\begin{defi}
	Let $D\subset\R^d$ be a domain of polyhedral type with singularity set $S\subset\partial D$, and put
	$\Omega=\R^d\setminus S$. Further, let $s\in\R_+$, $a\in\R$,  and $1<p<\infty$. Then we define
	$$\mathfrak{K}^s_{a,p}(D)
		=\bigl\{u:D\rightarrow\R\big|\ \exists v\in\mathfrak{K}^s_{a,p}(\Omega)\text{ with }v|_D=u\bigr\}$$
	with norm
	$$\|u|\mathfrak{K}^s_{a,p}(D)\|=\inf_{v|_D=u}\|v|\mathfrak{K}^s_{a,p}(\Omega)\|.$$
\end{defi}

The boundedness of Stein's extension operator now guarantees that this definition is consistent with the original one for Kondratiev spaces $\calk^m_{a,p}(D)$, i.e. we always have
$$\calk^m_{a,p}(D)=\mathfrak{K}^m_{a,p}(D)$$
for all parameters $m\in\N_0$, $a\in\R$,  and $1<p<\infty$. This is also the basis for

\begin{satz}
	Theorem \ref{thm-general-interpol} carries over mutatis mutandis upon replacing the domain $\Omega$ by the domain
	$D$.
\end{satz}

\begin{proof}
	This is an immediate consequence of the observation
	\begin{equation*}
		\bigl[\mathfrak{K}^{s_0}_{a_0,p}(D),\mathfrak{K}^{s_1}_{a_1,p}(D)\bigr]_\Theta
			=\bigl\{u:D\rightarrow\R\big|\ 
				\exists v\in [\mathfrak{K}^{s_0}_{a_0,p}(\Omega),\mathfrak{K}^{s_1}_{a_1,p}(\Omega)]_\Theta
				\text{ with }u=v|_D\bigr\},
	\end{equation*}
	correspondingly for the norms. That observation in turn can be seen in a number of ways. Either again by referring to the
	method of retractions and coretractions -- here the restriction operator
	$$\mathfrak{R}:\calk^m_{a,p}(\Omega)\rightarrow\calk^m_{a,p}(D),
		\qquad u\mapsto u|_D,$$
	is the retraction (it obviously is a bounded linear operator with norm $1$), and the Stein extension $\mathfrak{E}$ the
	corresponding coretraction (note $\mathfrak{R}\circ\mathfrak{E}=\text{Id}$ on $\calk^m_{a,p}(D)$ for all parameters
	$m$, $a$,  and $p$). An alternative argument would be given by referring to the abstract result from
	\cite[Theorem 1.17.2]{Tr78} concerning interpolation of quotient spaces -- the space
	$$U=\{u\in\calk^m_{a,p}(\Omega):u|_D=0\text{ on }D\}$$
	is a complemented subspace of $\calk^m_{a,p}(\Omega)$, with complement
	$$V=(\mathfrak{E}\circ\mathfrak{R})\calk^m_{a,p}(\Omega)=\text{Range}(\mathfrak{E}),$$
	thus $\calk^m_{a,p}(D)$ being isomorphic to $\calk^m_{a,p}(\Omega)/U$ (the canonical isomorphism being generated
	by $\mathfrak{E})$. Finally, a more direct argument (though essentially equivalent to the other two versions) can be
	given based on arguments in \cite[Theorem 2.13]{Tr02}.
\end{proof}




\section{An embedding result for fractional Kondratiev spaces}

As an application we shall now present and proof a Sobolev-type embedding result for Kondratiev spaces. A more explicit proof of this result for spaces with integer smoothness can be found in \cite[Section 3]{smcw-a}, based on earlier results by Mazja and Rossmann \cite{MR1, MR2}. Our proof will be more abstract, relying on complex interpolation and once more the relation to refined localization spaces.

We begin with a preparation, a Sobolev-type embedding result for refined localization spaces.

\begin{lem}
	Let $D\subset\R^d$ be an arbitrary domain, and let $0<p_0\leq p_1<\infty$, $0<q_0,q_1\leq\infty$,  and $s_0,s_1\in\R$,
	where
	$$
		s_0-\frac{d}{p_0}\geq s_1-\frac{d}{p_1}\,.
	$$
	Then we have an embedding
	$$F^{s_0,\mathrm{rloc}}_{p_0,q_0}(D)\hookrightarrow F^{s_1,\mathrm{rloc}}_{p_1,q_1}(D).$$
\end{lem}

\begin{proof}
	We first observe that under the exactly same conditions we have a continuous embedding
	$$F^{s_0}_{p_0,q_0}(\R^d)\hookrightarrow F^{s_1}_{p_1,q_1}(\R^d).$$
	The claim now follows immediately from the definition of the quasi-norm in refined localization spaces, since we have
	for all $j\geq 0$
	$$\|\varphi_j u|F^{s_1}_{p_1,q_1}(\R^d)\|\leq c\,\|\varphi_j u|F^{s_0}_{p_0,q_0}(\R^d)\|$$
	with a constant independent of $j\geq 0$ and $u\in F^{s_0,\mathrm{rloc}}_{p_0,q_0}(D).$
\end{proof}

\begin{satz}
	Let $D\subset\R^d$ be a domain of polyhedral type with singularity set $S\subset\partial D$, and put
	$\Omega=\R^d\setminus S$. Further, let $s_0,s_1\in\R_+$, $1<p_0\leq p_1<\infty$ and $a_0,a_1\in\R$, where
	$$
		s_0-\frac{d}{p_0}\geq s_1-\frac{d}{p_1}\qquad\text{and}\qquad a_0-\frac{d}{p_0}\geq a_1-\frac{d}{p_1}.
	$$
	Then we have continuous embeddings
	$$
		\mathfrak{K}^{s_0}_{a_0,p_0}(\Omega)\hookrightarrow\mathfrak{K}^{s_1}_{a_1,p_1}(\Omega)
		\qquad\text{as well as}\qquad
		\mathfrak{K}^{s_0}_{a_0,p_0}(D)\hookrightarrow\mathfrak{K}^{s_1}_{a_1,p_1}(D)
	$$
\end{satz}

\begin{proof}
	{\bf Step 1:} We first assume the stronger condition $s_0-s_1=d\bigl(\frac{1}{p_0}-\frac{1}{p_1}\bigr)$.
	Under this condition we have an embedding
	$$F^{s_0,\mathrm{rloc}}_{p_0,2}(\Omega)\hookrightarrow F^{s_1,\mathrm{rloc}}_{p_1,2}(\Omega),$$
	which in view of Theorem \ref{thm-kond-rloc} and Definition \ref{def-kond-fractional} is equivalent to
	$$\mathfrak{K}^{s_0}_{s_0,p_0}(\Omega)\hookrightarrow\mathfrak{K}^{s_1}_{s_1,p_1}(\Omega).$$
	Applying the isomorphism $T_{a_0-s_0}$ to this embedding results in
	$$\mathfrak{K}^{s_0}_{a_0,p_0}(\Omega)\hookrightarrow\mathfrak{K}^{s_1}_{s_1-s_0+a_0,p_1}(\Omega).$$
	Since under our assumptions we have
	$$a_1\leq a_0-d\Bigl(\frac{1}{p_0}-\frac{1}{p_1}\Bigr)=a_0-(s_0-s_1)$$
	the right-hand side space in that last embedding is itself embedded into $\mathfrak{K}^{s_1}_{a_1,p_1}(\Omega)$, 
	proving the claim in this special case.
	
	{\bf Step 2:} The general case now follows by observing
	$$\mathfrak{K}^{s_0}_{a_0,p_0}(\Omega)
		\hookrightarrow\mathfrak{K}^{\widetilde s_1}_{a_1,p_1}(\Omega)
		\hookrightarrow\mathfrak{K}^{s_1}_{a_1,p_1}(\Omega),$$
	where
	$$\widetilde s_1=s_0-d\Bigl(\frac{1}{p_0}-\frac{1}{p_1}\Bigr)\geq s_1$$
	fulfills the stronger condition from Step 1. In turn, the second embedding is an immediate consequence of Lemma
	\ref{lemma-embedding-vertical}.
	
	{\bf Step 3:} The result for Kondratiev spaces on $D$ then immediately follows from the definition of the spaces via
	restriction.
\end{proof}

Sharpness of these conditions can be verified most directly for the standard domains $\R^d\setminus\R^\ell_*$ from simple scaling arguments as in \cite[Proposition 3.1]{smcw-a}. We omit the details.

\begin{appendix}

\section{$u$-wavelet systems}

In his monograph \cite{Tr08} Triebel presented a construction of wavelet systems for arbitrary (bounded or unbounded) domains $\Omega$ which turned out to be orthonormal bases for $L_2(\Omega)$, unconditional bases for $L_p(\Omega)$, $1<p<\infty$, and representation systems for $F^{s,\mathrm{rloc}}_{p,q}(\Omega)$ for the whole range of parameters of interest to us, i.e. $0<p,q\leq\infty$ and $s>\sigma_{p,q}$. The construction is quite involved, so we will not present details here,  but instead focus on recalling the general concept and the resulting characterization for the refined localization spaces. For more details, a proof of the wavelet characterization stated below and further references, we refer to \cite[Chapter 2]{Tr08}.

We start with recalling the construction of a wavelet basis on $\R^d$ by tensorization of a univariate wavelet system. Let $\psi_0$ be a univariate scaling function and $\psi_1$ the associated wavelet, so that
$$\{\psi_{0,k}=\psi_0(\cdot-k):k\in\Z\}\cup\{\psi_{j,k}=2^{j/2}\psi_1(2^j\cdot -k):j\geq 0, k\in\Z\}$$
forms an orthonormal basis of $L_2(\R)$. Constructions of such scaling functions and wavelets with various desirable properties are well-known; let us particularly mention the construction by Daubechies for compactly supported wavelet systems of arbitrarily prescribed (finite) smoothness \cite{Daub}. Now let $E=\{0,1\}^d\setminus\{0\}$, and define
$$\Psi_0(x)=\prod_{j=1}^d\psi_0(x_j)\qquad\text{and}\qquad\Psi_e(x)=\prod_{j=1}^d\psi_{e_j}(x_j),\quad e\in E,$$
the $d$-variate scaling function and the associated $2^d-1$ wavelets. Then
$$\{\Psi_{0,k}=\psi_0(\cdot-k):k\in\Z^d\}\cup\{\Psi_{e,j,k}=2^{jd/2}\Psi_e(2^j\cdot -k):j\geq 0, k\in\Z^d\}$$
forms an orthonormal basis of $L_2(\R^d)$. Finally, we note that by replacing $\psi_i$ by $\psi_i^L=2^{L/2}\psi_i(2^L\cdot)$ for a given $L\in\N$ we can control (shrink) the support of the scaling function and wavelet as required in further constructions. The tensorized functions $\Psi^L_{e,j,k}$ are defined accordingly.

\begin{defi}
	Let $\Omega\subsetneq\R^d$ be an arbitrary domain, $\Gamma=\partial\Omega$. Then a set
	$$\Z_\Omega=\{x_{j,r}\in\Omega: j\in\N_0, r=1,\ldots,N_j\},\qquad N_j\in\overline{\N}=\N\cup\{\infty\},$$
	with
	$$|x_{j,r}-x_{j,r'}|\geq c_1 2^{-j},\quad j\in\N_0,r\neq r',$$
	as well as
	$$\dist\Bigl(\bigcup_{r=1}^{N_j}B(x_{j,r}, c_2 2^{-j}),\Gamma\Bigr)\geq c_3 2^{-j},\quad j\in\N_0$$
	for some positive real constants $c_1$, $c_2$,  and $c_3$ is called an approximate lattice in $\Omega$.
\end{defi}

\begin{defi}
	Let $\Omega\subsetneq\R^d$ be an arbitrary domain, and let $\Z_\Omega$ be an approximate lattice in $\Omega$. Let
	$K,L,u\in\N$, $D>0$,  and $c_4>0$. Then a system
	$$\{\Phi_{j,r}:j\in\N_0,1,\ldots,N_j\}\subset C^u(\R^d),\quad N_j\in\overline{\N},$$
	with
	$$\supp\Phi_{j,r}\subset B(x_{j,r}, c_2 2^{-j}),\quad j\in\N_0,$$
	is called a $u$-wavelet system (with respect to $\Omega$), if it consists of the following three types of functions:
	\begin{enumerate}
		\item
			Basic wavelets
			$$\Phi_{0,r}=\Psi^L_{e,0,m}\quad\text{for some}\quad e\in\{0,1\}^d,\  m\in\Z^d,$$
		\item
			Interior wavelets
			$$\Phi_{j,r}=\Psi^L_{e,j,m},\quad j\in\N,\  \dist(x_{j,r},\Gamma)\geq c_4 2^{-j},$$
			for some $e\in E$ and $m\in\Z^d$,
		\item
			Boundary wavelet
			$$\Phi_{j,r}=\sum_{|m-m'|\leq K}d^j_{m,m'}\Psi^L_{\underline e,j,m'},
				\quad j\in\N,\  \dist(x_{j,r},\Gamma) < c_4 2^{-j},$$
			for some $m\in\Z^d$ (depending on $j$ and $r$) and coefficients $d^j_{m,m'}\in\R$ with
			$$\sum_{|m-m'|\leq K}|d^j_{m,m'}|\leq D\quad\text{and}\quad
				\supp\Psi^L_{\underline e,j,m'}\subset B(x_{j,r}, c_2 2^{-j}).$$
			Therein $\underline e=(1,\ldots,1)\in\{0,1\}^d$.
	\end{enumerate}
\end{defi}

In \cite{Tr08} it was shown that for arbitrary domains and every given $u\in\N$ such $u$-wavelet systems can  indeed be constructed.

\begin{defi}
	Let $\Omega\subsetneq\R^d$ be an arbitrary domain, and let $\Z_\Omega$ be an approximate lattice in $\Omega$. Let
	$s\in\R$ and $0<p,q\leq\infty$. Then $b^s_{p,q}(\Z_\Omega)$ is the collection of all sequences
	$$\lambda=\{\lambda_{j,r}\in\mathbb{C}:j\in\N_0,r=1,\ldots,N_j\},\quad N_j\in\overline{\N},$$
	such that
	$$\|\lambda|b^s_{p,q}(\Z_\Omega)\|=\Biggl(\sum_{j\geq 0}2^{j(s-\frac{n}{p})q}
		\biggl(\sum_{r=1}^{N_j}|\lambda_{j,r}|^p\biggr)^{q/p}\Biggr)^{1/q}<\infty.$$
	Similarly, the set $f^s_{p,q}(\Z_\Omega)$ is the collection of all sequences $\lambda$ of the above form such that
	$$\|\lambda|f^s_{p,q}(\Z_\Omega)\|=\Biggl\|\biggl(\sum_{j\geq 0}\sum_{r=1}^{N_j}
		2^{jsq}|\lambda_{j,r}|^q\chi_{j,r}(\cdot)\biggr)^{1/q}\Bigg|L_p(\Omega)\Biggr\|<\infty,$$
	wherein $\chi_{j,r}$ denotes the characteristic function of the ball $B(x_{j,r},c_2 2^{-j})$.
\end{defi}

\end{appendix}


\end{document}